\documentclass[12pt]{amsart}
\usepackage{amsmath}
\usepackage{amsfonts}
\usepackage{amssymb}
\usepackage[all,cmtip]{xy}           %xypic macro for latex2.09

\usepackage{bbm}
\usepackage{bbding}
\usepackage{txfonts}
\usepackage[shortlabels]{enumitem}
\usepackage{ifpdf}
\ifpdf
\usepackage[colorlinks,final,backref=page,hyperindex]{hyperref}
\else
\usepackage[colorlinks,final,backref=page,hyperindex,hypertex]{hyperref}
\fi
\usepackage{tikz-cd}
\usepackage[active]{srcltx}
\usepackage{tikz}
\usepackage{amscd}
\usepackage{bm}

\makeatletter

\newtheorem{thm}{Theorem}[section]
\newtheorem{prop}[thm]{Proposition}
\newtheorem{lem}[thm]{Lemma}

\newtheorem{prop-def}{thm}[section]

\theoremstyle{definition}
\newtheorem{defn}[thm]{Definition}

\newtheorem{exam}[thm]{Example}

\newcommand{\nc}{\newcommand}

 \nc{\mbibitem}[1]{\bibitem{#1}} % Use this to show number
 \nc{\mrm}[1]{{\rm #1}}

\nc{\pl}{\cdot}
 \nc{\la}{\longrightarrow}
\nc{\ot}{\otimes}
 \nc{\rar}{\rightarrow}

\nc{\PLA}{{\mathrm{Leib}}}
\nc{\bfk}{{\bf k}}
\nc{\C}{{\mathrm{C}}}
\nc{\RBA}{\mathsf{MLeib}}
\nc{\RBO}{\mathsf{MO}}
\nc{\End}{\mrm{End}}
\nc{\Ext}{\mrm{Ext}}
\nc{\Fil}{\mrm{Fil}}
\nc{\Fr}{\mrm{Fr}}
\nc{\Frob}{\mrm{Frob}}
\nc{\Gal}{\mrm{Gal}}
\nc{\GL}{\mrm{GL}}
\nc{\Hom}{\mrm{Hom}}
\nc{\Hoch}{\mrm{Hoch}}
\nc{\hsr}{\mrm{H}}
\nc{\hpol}{\mrm{HP}}
\nc{\im}{\mrm{im}}
\nc{\Id}{\mrm{Id}}
\nc{\Irr}{\mrm{Irr}}
\nc{\incl}{\mrm{incl}}
\nc{\length}{\mrm{length}}
\nc{\NLSW}{\mrm{NLSW}}
\nc{\Lie}{\mrm{Lie}}
\nc{\mchar}{\rm char}
\nc{\mpart}{\mrm{part}}
\nc{\ql}{{\QQ_\ell}}
\nc{\qp}{{\QQ_p}}
\nc{\rank}{\mrm{rank}}
\nc{\rcot}{\mrm{cot}}
\nc{\rdef}{\mrm{def}}
\nc{\rdiv}{{\rm div}}
\nc{\rmH}{ {\mathrm{H}}}
\nc{\rtf}{{\rm tf}}
\nc{\rtor}{{\rm tor}}
\nc{\res}{\mrm{res}}
\nc{\Sh}{{\mathrm{Sh}}}
\nc{\SL}{\mrm{SL}}
\nc{\Spec}{\mrm{Spec}}
\nc{\sgn}{{\mathrm{sgn}}}
\nc{\tor}{\mrm{tor}}
\nc{\Tr}{\mrm{Tr}}
\nc{\tr}{\mrm{tr}}
\nc{\wt}{\mrm{wt}}
\nc{\op}{\mrm{op}}

\nc{\BA}{{\mathbb A}}   \nc{\CC}{{\mathbb C}}
\nc{\DD}{{\mathbb D}}   \nc{\EE}{{\mathbb E}}
\nc{\FF}{{\mathbb F}}   \nc{\GG}{{\mathbb G}}
\nc{\HH}{ \mathrm{HH}}   \nc{\LL}{{\mathbb L}}
\nc{\NN}{{\mathbb N}}   \nc{\PP}{{\mathbb P}}
\nc{\QQ}{{\mathbb Q}}   \nc{\RR}{{\mathbb R}}
\nc{\TT}{{\mathbb T}}   \nc{\VV}{{\mathbb V}}
\nc{\ZZ}{{\mathbb Z}}   \nc{\TP}{\widetilde{P}}
\nc{\m}{{\mathbbm m}}

\nc{\cala}{{\mathcal A}}    \nc{\calc}{{\mathcal C}}
\nc{\cald}{\mathcal{D}}     \nc{\cale}{{\mathcal E}}
\nc{\calf}{{\mathcal F}}    \nc{\calg}{{\mathcal G}}
\nc{\calh}{{\mathcal H}}    \nc{\cali}{{\mathcal I}}
\nc{\call}{{\mathcal L}}    \nc{\calm}{{\mathcal M}}
\nc{\caln}{{\mathcal N}}    \nc{\calo}{{\mathcal O}}
\nc{\calp}{{\mathcal P}}    \nc{\calr}{{\mathcal R}}
\nc{\cals}{{\mathcal S}}    \nc{\calt}{{\Omega}}
\nc{\calv}{{\mathcal V}}    \nc{\calw}{{\mathcal W}}
\nc{\calx}{{\mathcal X}}

\nc{\fraka}{{\mathfrak a}}
\nc{\frakb}{\mathfrak{b}}
\nc{\frakg}{{\frak g}}
\nc{\frakl}{{\frak l}}
\nc{\fraks}{{\frak s}}

\nc{\frakm}{{\frak m}}
\nc{\frakM}{{\frak M}}
\nc{\frakp}{{\frak p}}
\nc{\frakW}{{\frak W}}
\nc{\frakX}{{\frak X}}
\nc{\frakS}{{\frak S}}
\nc{\frakA}{{\frak A}}
\nc{\frakC}{{\frak{C}}}
\nc{\frakx}{{\frakx}}

\nc{\lir}[1]{\textcolor{red}{\underline{Li:}#1 }}

\begin{document}

\title[ Modified Rota-Baxter Leibniz algebras]{ Cohomology and Deformation theory of Modified Rota-Baxter Leibniz algebras}

\author{Yizheng Li}
\address{
 School of Mathematical Sciences, Qufu Normal University,  Qufu 273165, P. R. of China}
\email{yzli1001@163.com}

\author{$\mbox{Dingguo\ Wang}^1$}\footnote{Corresponding author}
\address{School of Mathematical Sciences, Qufu Normal University,  Qufu 273165, P. R. of China}
\email{dgwang@qfnu.edu.cn}

\date{\today}

\begin{abstract}
 In this paper, we define the cohomology of a modified Rota-Baxter Leibniz algebra with
coefficients in a suitable representation.  As applications of our cohomology, we study formal
one-parameter deformations and abelian extensions of modified Rota-Baxter  Leibniz  algebras.
\end{abstract}

\subjclass[2010]{
16E40   %(co)homology of rings and algebras
16S80   %deformations of rings
17A30  %Rota-Baxter pre-Lie algebra
}

\keywords{ Modified Rota-Baxter Leibniz algebra, cohomology,  deformation, Abelian extension}

\maketitle

\tableofcontents

\allowdisplaybreaks

\section*{Introduction}

In 1960, Baxter \cite{B60} introduced the notion of Rota-Baxter operators on associative algebras in his study of fluctuation theory in probability. Rota-Baxter operators have been found many applications, including in Connes-Kreimer's algebraic approach
to the renormalization in perturbative quantum field theory \cite{CK00}. More precisely.
For a fixed constant $\lambda$, a Rota-Baxter operator of weight $\lambda$ is a linear operator $P$ on
an associative algebra $R$ that satisfies the Rota-Baxter equation:
\begin{eqnarray}
 P(x)P(y) = P(P(x)y) + P(xP(y)) + \lambda P(xy), \forall x, y \in R.
\end{eqnarray}
For more details on the Rota-Baxter operator, see \cite{G12}.

Semonov-Tian-Shansky \cite{STS} found that, under
suitable conditions, the operator form of the classical Yang-Baxter equation is precisely the
Rota-Baxter operator of weight 0 on a Lie algebra. As a modified form of the operator
form of the classical Yang-Baxter equation, he also introduced in that paper the modified
classical Yang-Baxter equation whose solutions are called modified $r$-matrices, modified $r$-matrices have been found many applications, including applications in the study of generalized Lax pairs and affine geometry on Lie
groups.  Recently,  the authors have considered formal deformations of  modified $r$-matrices motivated by the deformations of classical $r$-matrices in \cite{JS}.

A Rota-Baxter operator  gives a modified Rota-Baxter operator  through
a linear transformation \cite{E02}.  Guo et al. studied the
Rota-Baxter type algebra and constructed the free object uniformly in that category in \cite{ZGG}.  Recently, the author has obtained  the cohomology and  formal deformations of  modified Rota-Baxter algebras in \cite{Das2}.

The concept of Leibniz algebras was introduced  by Loday \cite{L93, LP93} in the study of the algebraic $K$-theory. Relative Rota-Baxter
operators on Leibniz algebras were studied in \cite{ST2}, which is the main ingredient in the study of the twisting theory and the bialgebra theory for Leibniz algebras \cite{ST1}. Recently,  the author has considered  weighted relative Rota-Baxter operators  on Leibniz algebras in \cite{Das1}.  Our aim in this paper is to consider modified Rota-Baxter Leibniz algebras, and define the cohomology of a modified Rota-Baxter Leibniz algebra with
coefficients in a suitable representation.  As applications of our cohomology, we study formal
one-parameter deformations and abelian extensions of modified Rota-Baxter  Leibniz  algebras.

The paper is organized as follows. In Section 1, we recall some basic definitions about  Leibniz algebras and their cohomology. In Section 2, we consider  modified Rota-Baxter Leibniz algebras  and introduce their representations.  In Section 3, we define the cohomology of a modified Rota-Baxter Leibniz algebra with
coefficients in a suitable representation. In Section 4, we study formal deformation theory and rigidity of modified Rota-Baxter Leibniz
algebras. Finally, abelian extensions are considered in Section 5.

 Throughout this paper, let $\bfk$ be a field of characteristic $0$.  Except specially stated,  vector spaces are  $\bfk$-vector spaces and  all    tensor products are taken over $\bfk$.

\section{Leibniz algebras, representations and their cohomology theory}\
We start with the background of Leibniz algebras and their cohomology that we refer the reader to~ \cite{cuvier,LP93} for more details.

\begin{defn}
A Leibniz algebra is  a vector space $\frakg$ together with a bilinear operation (called bracket) $[\cdot, \cdot]_\frakg : \frakg \otimes \frakg \rightarrow  \frakg$  satisfying
\begin{eqnarray*}
[x, [y, z]_\frakg]_\frakg=[[x, y]_\frakg, z]_\frakg+[y, [x, z]_\frakg]_\frakg,~ \text{ for } x, y, z \in \frakg.
\end{eqnarray*}
\end{defn}

A Leibniz algebra as above may be denoted by the pair $(\frakg, [\cdot, \cdot])$ or simply by $\frakg$ when no confusion arises. A Leibniz algebra whose bilinear bracket is skewsymmetric is nothing but a Lie algebra. Thus, Leibniz algebras are the non-skewsymmetric analogue of Lie algebras.

\begin{defn}
A representation of a Leibniz algebra $(\frakg, [\cdot, \cdot]_\frakg)$ consists of a triple $(V, \rho^L, \rho^R)$ in which $V$ is a vector space and $\rho^L, \rho^R : \frakg \rightarrow gl(V)$ are linear maps satisfying for $x, y \in \frakg$,
\begin{align*}
\begin{cases}
\rho^L ([x,y]_\frakg) = \rho^L (x)  \rho^L (y) - \rho^L (y)  \rho^L (x),\\
\rho^R ([x,y]_\frakg) = \rho^L (x) \rho^R (y) - \rho^R (y)  \rho^L (x),\\
\rho^R ([x,y]_\frakg) = \rho^L (x)  \rho^R (y) + \rho^R (y)  \rho^R (x).
\end{cases}
\end{align*}
\end{defn}

It follows that any Leibniz algebra $\frakg$ is a representation of itself with
\begin{align*}
    \rho^L (x) = L_x = [x, ~]_\frakg ~~ \text{ and } ~~ \rho^R (x) = R_x = [~, x]_\frakg, \text{ for } x \in \frakg.
\end{align*}
Here $L_x$ and $R_x$ denotes the left and right multiplications by $x$, respectively. This is called the regular representation.

\medskip

Let $(\frakg, [\cdot, \cdot]_\frakg)$ be a Leibniz algebra and  $(V, \rho^L, \rho^R)$ be a representation of it. The cohomology of the Leibniz algebra $\frakg$ with coefficients in $V$ is the
cohomology of the cochain complex $\{ C^*(\frakg, V), \delta \}$, where $C^n (\mathfrak{g}, V) = \mathrm{Hom}(\frakg^{\otimes n}, V )$ for $n \geq 0$, and the
coboundary operator $\delta:C^n(\frakg, V ) \rightarrow C^{n+1}(\frakg, V )$ given by
\begin{align*}
&(\delta^n f)(x_1,\ldots, x_{n+1})\\
&=\sum^{n}_{i=1}(-1)^{i+1}\rho^L(x_i)f(x_1,\ldots, \hat{x}_i, \ldots, x_{n+1}) + (-1)^{n+1}\rho^R(x_{n+1})f(x_1,\ldots,  x_{n})\\
&+\sum_{1\leq i< j\leq n+1} (-1)^{i}f(x_1,\ldots, \hat{x}_i, \ldots, x_{j-1}, [x_i, x_j]_\frakg, x_{j+1},\ldots, x_{n+1}),
\end{align*}
for $x_1, \ldots, x_{n+1}\in \frakg$. The corresponding cohomology groups are denoted by $H^*_{\mathrm{Leib}}(\frakg, V).$ This cohomology has been first appeared in \cite{cuvier} and rediscover by Loday and Pirashvili \cite{LP93}. This cohomology is also the Loday-Pirashvili cohomology.

% Let $V$ be a graded vector space.
%   Let $$T(sV)=k\oplus sV\oplus (sV)^{\otimes 2}\oplus \cdots.$$
%   Let $$\mathfrak{C}_{Alg}(V)=\mathrm{Hom}(T(sV), sV).$$
%   One can generalise the construction of Gerstenhaber and obtain a graded Lie algebra structure on $\mathfrak{C}_{Alg}(V)$.

\section{Representations of  modified Rota-Baxter Leibniz algebras}

In  this section, we consider  modified Rota-Baxter Leibniz algebras  and introduce their representations.
We also provide various examples and new constructions.

\begin{defn}\label{Def: Rota-Baxter 3-Lie algebra}
Let $\lambda\in \bfk$. Given  a Leibniz algebra  $(\frakg,  \mu=[\cdot, \cdot]_\frakg)$, it is called  a  modified Rota-Baxter Leibniz algebra of weight $\lambda$,  if there is a modified Rota-Baxter
operator $K: \frakg \rightarrow \frakg $ of weight $\lambda$ such that
\begin{eqnarray}\label{Eq: Rota-Baxter relation}
[K(x),  K(y)]_\frakg&=&K\big([K(x), y]_\frakg+[x, K(y)]_\frakg)+\lambda[x, y]_\frakg\end{eqnarray}	
for any $x, y \in \frakg $.

 A modified Rota-Baxter Leibniz algebra of weight $\lambda$ is a triple $(\frakg, [\cdot, \cdot]_\frakg,  K)$ consisting of a Leibniz algebra $(\frakg,  [\cdot, \cdot]_\frakg)$ together with a
modified Rota-Baxter operator $K$ of weight $\lambda$ on it.
\end{defn}
\begin{exam}
Consider the three-dimensional Leibniz algebra $(\mathfrak{g}, [\cdot, \cdot]_\mathfrak{g})$ given with
respect to a basis $\{e_1, e_2, e_3\}$ whose nonzero products are given as follows:
\begin{eqnarray*}
[e_1, e_1]=e_3.
\end{eqnarray*}
Then $K=\left(
                                            \begin{array}{ccc}
                                              a_{11} & a_{12} & a_{13} \\
                                              a_{21} & a_{22} & a_{23} \\
                                              a_{31} & a_{32} & a_{33} \\
                                            \end{array}
                                          \right)$        is a  modified Rota-Baxter
operator of weight $\lambda$ on $(\mathfrak{g}, [\cdot, \cdot]_\mathfrak{g})$  if and only if
\begin{eqnarray*}
&&[Ke_i,Ke_j]_\mathfrak{g}= K([K e_i, e_j ]_\mathfrak{g}+[e_i, Ke_j]_\mathfrak{g})+\lambda[e_i,e_j]_\mathfrak{g},    ~~~~~\forall i, j=1, 2, 3.
\end{eqnarray*}
We have $[K e_1,Ke_1]_\mathfrak{g}=[a_{11}e_1 + a_{21}e_2 + a_{31}e_3, a_{11}e_1 + a_{21}e_2 + a_{31}e_3]_\mathfrak{g}= a^2_{11}e_3$ and
\begin{eqnarray*}
&&K ([K e_1, e_1]_\mathfrak{g}+[e_1, Ke_1]_g)+\lambda[e_1,e_1]_\mathfrak{g}\\
&=&K ([a_{11}e_1 + a_{21}e_2 + a_{31}e_3, e_1]_\mathfrak{g}+[e_1, a_{11}e_1 + a_{21}e_2 + a_{31}e_3]_\mathfrak{g})+\lambda e_3\\
&=& 2a_{11}Ke_3+\lambda e_3\\
&=& 2a_{11}a_{13}e_1+2a_{11}a_{23}e_2+2a_{11}a_{33}e_3+\lambda e_3.
\end{eqnarray*}
Thus, by $[K e_1,Ke_1]_\mathfrak{g}=K ([K e_1, e_1]_\mathfrak{g}+[e_1, Ke_1]_\mathfrak{g})+\lambda[e_1,e_1]_\mathfrak{g}$, we have
\begin{eqnarray*}
 a_{11}a_{13}=0, a_{11}a_{23}=0, a^2_{11}=2a_{11}a_{33}+\lambda.
\end{eqnarray*}
Similarly, we obtain
\begin{eqnarray*}
&& a_{11}a_{12}=a_{12}a_{33},~~~~~~~~~~~ a_{12}a_{13}=0,~~~~~~~~~~ a_{12}a_{23}=0;\\
&& a_{11}a_{13} = a_{13}a_{33},~~~~~~~~~~~ a_{13}a_{13} = 0,~~~~~~~~~~ a_{13}a_{23} = 0;\\
&& a_{12}a_{11} =  a_{12}a_{33}, ~~~~~~~~~~~a_{12}a_{13} = 0, ~~~~~~~~~~a_{12}a_{23} = 0;\\
&& a_{13}a_{11} =  a_{13}a_{33},~~~~~~~~~~~ a_{13}a_{13} = 0,~~~~~~~~ ~~ a_{13}a_{23} = 0,\\
&& a^2_{12} = 0, ~~~~~~~~~~~~~~~~~~~~~~~ a^2_{13} = 0,~~~~~~~~~~~~~~~~~~~~~~~~~  a_{12}a_{13} = 0.
\end{eqnarray*}
Summarize the above discussion, we have the following two cases:

(1)~If $a_{11}=a_{12}=a_{13}=0$,  then any $K=\left(
                                            \begin{array}{ccc}
                                              0 & 0 & 0 \\
                                              a_{21} & a_{22} & a_{23} \\
                                              a_{31} & a_{32} & a_{33} \\
                                            \end{array}
                                          \right)$            is a modified Rota-Baxter
operator of weight 0  on $(\mathfrak{g}, [\cdot, \cdot]_\mathfrak{g})$.

(2)~ If $a_{12}=a_{13}=a_{23}=0$ and $a_{11} \neq 0$,  then any  $K=\left(
                                            \begin{array}{ccc}
                                              a_{11} & 0 & 0 \\
                                              a_{21} & a_{22} & 0 \\
                                              a_{31} & a_{32} & a_{33} \\
                                            \end{array}
                                          \right)$      is a  modified Rota-Baxter
operator of weight $(a^2_{11}-2a_{11}a_{33})$ on $(\mathfrak{g}, [\cdot, \cdot]_\mathfrak{g})$.
\end{exam}
\begin{defn}
Let $(\frakg, [\cdot, \cdot]_\frakg,  K)$ and $(\frakg', [\cdot, \cdot]_{\frakg'},  K')$ be two modified Rota-Baxter Leibniz algebras of weight $\lambda$. A morphism from
$(\frakg, [\cdot, \cdot]_\frakg,  K)$ to $(\frakg', [\cdot, \cdot]_{\frakg'},  K')$ is a Leibniz algebra homomorphism $\phi: \frakg\rightarrow \frakg'$  satisfying  $\phi \circ K=K'\circ \phi$.

\end{defn}

Denote by $\RBA$ the category of modified Rota-Baxter Leibniz algebras of weight $\lambda$ with obvious morphisms.

Recall  that  a linear map $T: \frakg\rightarrow \frakg$ is said to be  a Rota-Baxter operator of  weight $\lambda$ on the  Leibniz algebra $(\frakg, [\cdot, \cdot]_\frakg)$ if $T$ satisfies
\begin{eqnarray*}
[T(x), T(y)]_\frakg=T([T(x), y]_\frakg+[x, T(y)]_\frakg+\lambda[x, y]_\frakg), \text{~for~} x, y\in \frakg.
\end{eqnarray*}
A Rota-Baxter Leibniz algebra of weight $\lambda$ is a triple $(\frakg, [\cdot, \cdot]_\frakg,  T)$ consisting of a Leibniz algebra $(\frakg,  [\cdot, \cdot]_\frakg)$ together with a
Rota-Baxter operator $T$ of weight $\lambda$ on it.

Furthermore, we have
\begin{eqnarray*}
&&[(2T+\lambda \mathrm{id_\frakg})(x),  (2T+\lambda \mathrm{id_\frakg})(y)]_\frakg\\
&=& [2T(x)+\lambda x,  2T(y)+\lambda y]_\frakg\\
&=& 4[T(x), T(y)]_\frakg+2\lambda[T(x), y]_\frakg+2\lambda[x, T(y)]_\frakg+\lambda^2[x, y]_\frakg\\
&=&  4T([T(x), y]_\frakg+[x, T(y)]_\frakg+\lambda[x, y]_\frakg)\\
&& +2\lambda[T(x), y]_\frakg+2\lambda[x, T(y)]_\frakg+\lambda^2[x, y]_\frakg\\
&=&(2T+\lambda \mathrm{id_\frakg})\big([(2T+\lambda \mathrm{id_\frakg})(x), y]_\frakg+[x, (2T+\lambda \mathrm{id_\frakg})(y)]_\frakg)-\lambda^2[x, y]_\frakg.
\end{eqnarray*}
Thus, $(\frakg, [\cdot, \cdot]_\frakg,  2T+\lambda \mathrm{id_\frakg})$ is a modified Rota-Baxter Leibniz algebra of $-\lambda^2$.

According to the above discussion, we have
\begin{lem}
Let $(\frakg, [\cdot, \cdot]_\frakg)$ be a Leibniz algebra. Then $(\frakg, [\cdot, \cdot]_\frakg, T)$ is a Rota-Baxter Leibniz algebra of  weight $\lambda$ if and only if $(\frakg, [\cdot, \cdot]_\frakg, 2T+\lambda \mathrm{id_\frakg})$ is a modified Rota-Baxter Leibniz algebra of weight $-\lambda^2$.
\end{lem}
\begin{defn}\label{Def: Rota-Baxter representations} Let $(\frakg, [\cdot, \cdot]_\frakg, K)$ be a modified Rota-Baxter Leibniz algebra of weight $\lambda$ and $(V, \rho^L, \rho^R)$ be a representation over Leibniz
algebra $(\frakg, [\cdot, \cdot]_\frakg)$. We say that $(V, \rho^L, \rho^R,  K_V)$ is a representation over modified Rota-Baxter Leibniz algebra $(\frakg, [\cdot, \cdot]_\frakg, K)$ of weight $\lambda$   if $(V, \rho^L, \rho^R)$ is endowed with a linear operator $K_V : V \rightarrow V$ such that the following equations
\begin{eqnarray}\label{Def: Rota-Baxter representations}
&&\rho^L(K(x))K_V(v) =K_V\big(\rho^L(K(x)) v +\rho^L(x)K_V(v)\big)+\lambda\rho^L(x)v\nonumber\\
&&\rho^R(K(x))K_V(v) =K_V\big(\rho^R(K(x)) v +\rho^R(x)K_V(v)\big)+\lambda\rho^R(x)v
\end{eqnarray}
hold for any  $x \in \frakg$ and $v \in V$.

\end{defn}
\begin{exam}
It is ready to see that  $(\frakg ,[\cdot, \cdot]_\frakg, K)$ itself is a representation over the modified Rota-Baxter Leibniz algebra $(\frakg ,[\cdot, \cdot]_\frakg, K)$ of weight $\lambda$.
\end{exam}
\begin{defn}\cite{Das1}
 Let $(\frakg, [\cdot, \cdot]_\frakg, T)$ be a  Rota-Baxter Leibniz algebra of weight $\lambda$ and $(V, \rho^L, \rho^R)$ be a representation over Leibniz
algebra $(\frakg, [\cdot, \cdot]_\frakg)$. We say that $(V, \rho^L, \rho^R,  T_V)$ is a representation over  Rota-Baxter Leibniz algebra $(\frakg, [\cdot, \cdot]_\frakg, T)$ of weight $\lambda$   if $(V, \rho^L, \rho^R)$ is endowed with a linear operator $T_V : V \rightarrow V$ such that the following equations
\begin{eqnarray}\label{Def: Rota-Baxter representations}
&&\rho^L(T(x))T_V(v) =T_V\big(\rho^L(T(x)) v +\rho^L(x)T_V(v)+\lambda\rho^L(x)v\big)\nonumber\\
&&\rho^R(T(x))T_V(v) =T_V\big(\rho^R(T(x)) v +\rho^R(x)T_V(v)+\lambda\rho^R(x)v\big)
\end{eqnarray}
hold for any  $x \in \frakg$ and $v \in V$.

\end{defn}
The following result finds the relation between representation over Rota-Baxter Leibniz algebras and over modified
Rota-Baxter  Leibniz algebras.
\begin{prop}
Let $(\frakg, [\cdot, \cdot]_\frakg, T)$ be a Rota-Baxter Leibniz algebra of weight $\lambda$ and $(V, \rho^L, \rho^R, T_V)$ be a representation over Leibniz
algebra $(\frakg, [\cdot, \cdot]_\frakg)$. Then $(V, \rho^L,  \rho^R,  2T_V+\lambda \mathrm{id}_M)$ is   a representation  over the modified
Rota-Baxter  Leibniz algebras $(\frakg, [\cdot, \cdot]_\frakg, 2T+\lambda \mathrm{id_\frakg})$ of weight $-\lambda^2$.
\end{prop}
\begin{proof}
For any $x\in \frakg$ and $v\in V$, we have
\begin{eqnarray*}
&& \rho^L(2T(x)+\lambda x)(2T_V (v)+\lambda v)\\
&=& 4\rho^L(T(x))T_V (v)+2\lambda(\rho^L(T(x))v+\rho^L(x)T_V(v))+\lambda^2\rho^L(x)v\\
&=& 4T_V\big(\rho^L(T(x)) v +\rho^L(x)T_V(v)+\lambda\rho^L(x)v\big)\\
&&+2\lambda(\rho^L(T(x))v+\rho^L(x)T_V(v))+\lambda^2\rho^L(x)v\\
&=&(2T_V+\lambda \mathrm{id}_M)\big(\rho^L((2T+\lambda \mathrm{id_\frakg})(x)) v +\rho^L(x)(2T_V+\lambda \mathrm{id}_M)(v)\big)-\lambda^2\rho^L(x)v.
\end{eqnarray*}
Similarly, we can show that
\begin{eqnarray*}
&& \rho^R(2T(x)+\lambda x)(2T_V (v)+\lambda v)\\
&=&(2T_V+\lambda \mathrm{id}_M)\big(\rho^R((2T+\lambda \mathrm{id_\frakg})(x)) v +\rho^R(x)(2T_V+\lambda \mathrm{id}_M)(v)\big)-\lambda^2\rho^R(x)v.
\end{eqnarray*}
And completed the proof.
\end{proof}

In the following, we construct the semidirect product in the context of  modified Rota-Baxter Leibniz
algebras of any weight.

\begin{prop}
Let $(\frakg ,[\cdot, \cdot]_\frakg, K)$ be a   modified Rota-Baxter Leibniz algebra of weight $\lambda$  and $(V, \rho^L, \rho^R,  K_V)$  be a representation of it. Then $(\frakg\oplus V, K\oplus K_V )$ is a  modified Rota-Baxter Leibniz algebra, where the Leibniz bracket on $\frakg\oplus V$ is given
by the semidirect product
\begin{eqnarray}
 [x+u, y+v]_{\frakg\oplus V}&=&[x, y]_\frakg +\rho^L(x)v+\rho^R(y)u,
\end{eqnarray}
for any  $x, y\in \frakg$ and $u, v\in V$.
\end{prop}
\begin{proof}
For any  $x, y \in \frakg$ and $u, v\in V$, we have
\begin{eqnarray*}
&& [(K\oplus K_V)(x+u), (K\oplus K_V)(y+v)]_{\frakg\oplus V}\\
&=& [K(x), K(y)]_\frakg +\rho^L(K(x))K_V(v)+\rho^R(K(y))K_V(u)\\
&=&K\big([K(x), y]_\frakg+[x, K(y)]_\frakg\big)+\lambda[x, y]_\frakg +K_V\big(\rho^L(K(x)) v +\rho^L(x)K_V(v)\big)\\
&&+\lambda \rho^L(x)v+K_V\big(\rho^R(K(y)) u +\rho^R(y)K_V(u)\big) +\lambda\rho^R(y)u\\
&=& (K\oplus K_V)\big([(K\oplus K_V)(x+u), y+v]_{\frakg\oplus V}+[x+u,  (K\oplus K_V)(y+v)]_{\frakg\oplus V}\big)\\
&&+\lambda [x+u, y+v]_{\frakg\oplus V}
\end{eqnarray*}
This shows that  $K\oplus K_V$ is a  modified Rota-Baxter operator of  weight $\lambda$ on the semidirect product Leibniz algebra. Hence
the result follows.
\end{proof}

\begin{prop}\label{Prop: new RB 3-Lie algebra}
	Let $(\frakg, [\cdot, \cdot]_\frakg, K)$ be a modified Rota-Baxter Leibniz algebra of weight $\lambda$. Define a new binary operation as:
\begin{eqnarray}
[x, y]_K:&=&[K(x), y]_\frakg+[x, K(y)]_\frakg,
\end{eqnarray}
for any $x, y\in \frakg $. Then
\begin{itemize}

\item[(i)]      $(\frakg ,[\cdot, \cdot]_K )$ is a new  Leibniz algebra, we denote this Leibniz algebra by $\frakg_K$;

    \item[(ii)] the triple  $(\frakg_K ,[\cdot, \cdot]_K, K)$ is a modified Rota-Baxter Leibniz algebra of weight $\lambda$.
    \end{itemize}
	\end{prop}
\begin{proof}
(i) For any $x, y, z\in \frakg$, we have
\begin{align*}
[x, [y, z]_K]_K &= [x, [K(y), z]_\frakg+[y, K(z)]_\frakg]_K\\
&= [K(x), [K(y), z]_\frakg]_\frakg+ [x, K[K(y), z]_\frakg]_\frakg+ [K(x), [y, K(z)]_\frakg]_\frakg+[x, K[y, K(z)]_\frakg]_\frakg\\
&=[K(x), [K(y), z]_\frakg]_\frakg+ [K(x), [y, K(z)]_\frakg]_\frakg +[x, [K(y), K(z)]_\frakg]_\frakg-\lambda [x, [y,z]_\frakg]_\frakg\\
&= [[K(x), K(y)]_\frakg, z]_\frakg+[K(y), [K(x), z]_\frakg ]_\frakg +[[K(x), y]_\frakg, K(z)]_\frakg+[y, [K(x), K(z)]_\frakg ]_\frakg\\
&+ [[x, K(y)]_\frakg, K(z)]_\frakg+[K(y), [x, K(z)]_\frakg ]_\frakg -\lambda [[x, y]_\frakg z]_\frakg-\lambda [y, [x, z]_\frakg]_\frakg\\
&= [[K(x), y]_\frakg+[x, K(y)]_\frakg, z]_\frakg + [K(y), [K(x), z]_\frakg ]_\frakg +[[K(x), y]_\frakg, K(z)]_\frakg\\
&+[y, [K(x), z]_\frakg+[x, K(z)]_\frakg ]_\frakg + [[x, K(y)]_\frakg, K(z)]_\frakg+[K(y), [x, K(z)]_\frakg ]_\frakg\\
&= [[x, y]_K, z]_K+[y, [x, z]_K]_K.
\end{align*}
This shows that $(\frakg, [\cdot, \cdot]_K )$ is a  Leibniz algebra.

(ii)  For any $x, y \in \frakg $, we have
\begin{align*}
[K(x),  K(y)]_K&= [K^2(x), K(y)]_\frakg+[K(x), K^2(y)]_\frakg\\
&= K\big([K^2(x), y]_\frakg+[K(x), K(y)]_\frakg)+\lambda[K(x), y]_\frakg \\
&+ K\big([K(x), K(y)]_\frakg+[x, K^2(y)]_\frakg)+\lambda[x, K(y)]_\frakg \\
&=K([K(x), y]_K+[x, K(y)]_K)+\lambda [x, y]_K.
\end{align*}
This shows that the triple  $(\frakg_K ,[\cdot, \cdot]_K, K)$ is a modified Rota-Baxter Leibniz algebra of weight $\lambda$.
\end{proof}

The following result will be useful in the next section to construct
the cohomology of modified Rota-Baxter Leibniz algebras of weight $\lambda$.

\begin{thm}\label{Prop: new-representation}
Let $(\frakg ,[\cdot, \cdot]_\frakg, K)$ be a  modified Rota-Baxter Leibniz algebra of weight $\lambda$ and $(V, \rho^L, \rho^R, K_V)$  be a representation of it. Define maps $\rho_K^L, \rho_K^R: \frakg \rightarrow$ End$(V)$ by
\begin{eqnarray}
&&\rho^L_K(x)v:=\rho^L(K(x)) v-K_V(\rho^L(x)v),\nonumber\\
&& \rho^R_K(x)v:=\rho^R(K(x)) v-K_V(\rho^R(x)v)
\end{eqnarray}
for any $x\in \frakg, v\in V$. Then $\rho^L_K, \rho^R_K$  defines a representation of the Leibniz algebra $(\frakg_K ,[\cdot, \cdot]_K)$ on $V$. Moreover, $(V, \rho_K^L, \rho_K^R, K_V)$ is a representation
of the modified Rota-Baxter Leibniz algebra $(\frakg_K ,[\cdot, \cdot]_K, K)$ of weight $\lambda$.

\end{thm}
\begin{proof} One can show that  $\rho_K^L, \rho_K^R$  define a representation of the Leibniz algebra $(\frakg_K ,[\cdot, \cdot]_K)$ on $V$. In fact, we have
\begin{eqnarray*}
&&\rho_K^L(x) \rho_K^L(y)v-\rho_K^L(y) \rho_K^L(x)v\\
&=& \rho_K^L(x)(\rho^L(K(y)) v-K_V(\rho^L(y)v))-\rho_K^L(y)(\rho^L(K(x)) v-K_V(\rho^L(x)v))\\
&=& \rho^L(K(x))(\rho^L(K(y)) v-K_V(\rho^L(y)v))-K_V(\rho^L(x)(\rho^L(K(y)) v-\rho^L(x)K_V(\rho^L(y)v))\\
&& \rho^L(K(y))(\rho^L(K(x)) v-K_V(\rho^L(x)v))-K_V(\rho^L(y)(\rho^L(K(x)) v-\rho^L(y)K_V(\rho^L(x)v))\\
&=& \rho^L(K(x))\rho^L(K(y))v-K_V(\rho^L(K(x))\rho^L(y)v+\rho^L(x)K_V(\rho^L(y)v)-K_V(\rho^L(x)\rho^L(y)v))\\
&&-K_V(\rho^L(x)\rho^L(K(y))v-\rho^L(x)K_V(\rho^L(y)v))-\rho^L(K(y))\rho^L(K(x))v\\
&&+K_V(\rho^L(K(y))\rho^L(x)v+\rho^L(y)K_V(\rho^L(x)v)-K_V(\rho^L(y)\rho^L(x)v))\\
&&-K_V(\rho^L(y)\rho^L(K(x))v-\rho^L(y)K_V(\rho^L(x)v))\\
&=& \rho^L([K(x), K(y)]_\frakg)v-K_V(\rho^L([K(x), y]_\frakg+[x, K(y)]_\frakg)v)-\lambda\rho^L([x, y]_\frakg)\\
&=& \rho^L([x, y]_K)v-K_V(\rho^L([x, y]_K)v)\\
&=& \rho^L_K([x, y]_K)v.
\end{eqnarray*}
Similar to prove that
\begin{eqnarray*}
\rho_K^R([x, y]_K)=\rho_K^L (x)  \rho_K^R (y) - \rho_K^R (y)  \rho_K^L (x).
\end{eqnarray*}
 Furthermore, we have
\begin{eqnarray*}
 &&\rho_K^R(y) \rho_K^L(x)v+\rho_K^R(y) \rho_K^R(x)v\\
 &=&  \rho_K^R(y)(\rho^L(K(x)) v-K_V(\rho^L(x)v))+\rho_K^R(y)(\rho^R(K(x)) v-K_V(\rho^R(x)v) )\\
 &=& \rho^R(K(y))(\rho^L(K(x)) v-K_V(\rho^L(x)v))-K_V(\rho^R(y)(\rho^L(K(x)) v-K_V(\rho^L(x)v)))\\
 &&+\rho^R(K(y))(\rho^R(K(x)) v-K_V(\rho^R(x)v))-K_V(\rho^R(y)(\rho^R(K(x)) v-K_V(\rho^R(x)v)))\\
 &=& 0.
\end{eqnarray*}
 Moreover, we have
\begin{eqnarray*}
&& \rho^L_K(K(x))K_{V}(v)\\
&=& \rho^L(K^2(x))K_{V}(v)-K_V(\rho^L(K(x)) K_{V}(V))\\
&=& K_V\big(\rho^L(K^2(x)) v +\rho^L(N(x))N_V(v)\big)+\lambda\rho^L(K(x))v\\
&&-K^2_V\big(\rho^L(K(x)) v +\rho^L(x)K_V(v)\big)-\lambda K_V(\rho^L(x)v)\\
&=& K_V\big(\rho^L_K(K(x)) v +\rho^L_K(x)K_V(v)\big)+\lambda\rho^L_K(x)v.
\end{eqnarray*}
Therefore,  $(V, \rho_K^L, \rho_K^R, K_V)$ is a representation
of the  modified Rota-Baxter Leibniz algebra $(\frakg_K , [\cdot, \cdot]_K, K)$.
\end{proof}

\bigskip

\section{Cohomology theory of  modified Rota-Baxter Leibniz algebras} \label{Sect: Cohomology theory of Rota-Baxter 3-Lie algebras}
In this  section, we will define a cohomology theory for   modified Rota-Baxter Leibniz algebras.

\subsection{Cohomology of modified Rota-Baxter  operators}\
\label{Subsect: cohomology RB operator}

Firstly, let's study the cohomology of modified Rota-Baxter operators.

 Let $(\frakg , [\cdot, \cdot]_\frakg , K)$ be a modified Rota-Baxter Leibniz algebra and $(V, \rho^L, \rho^R, K_V)$ be a representation over it. Recall that
Proposition~\ref{Prop: new RB 3-Lie algebra}  and Proposition~\ref{Prop: new-representation}  give a new
Leibniz algebra   $\frakg_K$ and
  a new  representation  $V$ over $\frakg_K $.
 Consider the  cochain complex of $\frakg _K $ with
 coefficients in $V$:
 $$\C^\bullet_{\PLA}(\frakg _K, V)=\bigoplus\limits_{n=0}^\infty \C^n_{\PLA}(\frakg_K , V).$$
  More precisely,  for $n\geqslant 0$,  $ \C^n_{\PLA}(\frakg _K , V)=\Hom  (\mathfrak{g}^{\otimes n}, V)$ and its differential $$\partial^n:
 \C^n_{\mathrm{Leib}}(\frakg _K,\  V)\rightarrow  \C^{n+1}_{\mathrm{Leib}}(\frakg _K , V) $$ is defined as:
\begin{align*}
&(\partial^n f)(x_1,\ldots, x_{n+1})\\
&=\sum^{n}_{i=1}(-1)^{i+1}\rho^L_K(x_i)f(x_1,\ldots, \hat{x}_i, \ldots, x_{n+1}) + (-1)^{n+1}\rho^R_K(x_{n+1})f(x_1,\ldots,  x_{n})\\
&+\sum_{1\leq i< j\leq n+1} (-1)^{i}f(x_1,\ldots, \hat{x}_i, \ldots, x_{j-1}, [x_i, x_j]_K, x_{j+1},\ldots, x_{n+1}),
\end{align*}
for $x_1, \ldots, x_{n+1}\in \frakg$.
 \smallskip

 \begin{defn}
 	Let $(\frakg , [\cdot, \cdot]_\frakg , K)$ be a modified Rota-Baxter Leibniz algebra and $(V, \rho^L, \rho^R, K_V)$ be a representation over it.  Then the cochain complex $(\C^\bullet_\PLA(\frakg _K,  V),\partial)$ is called the cochain complex of modified Rota-Baxter operator $K$ with coefficients in $(V, \rho^L_K, \rho^R_K, K_V)$,  denoted by $C_{\RBO}^\bullet(\frakg , V)$. The cohomology of $C_{\RBO}^\bullet(\frakg ,V)$, denoted by $\mathrm{H}_{\RBO}^\bullet(\frakg ,V)$, are called the cohomology of  modified Rota-Baxter operator $K$ with coefficients in $(V, \rho^L_K, \rho^R_K, K_V)$.
 	
 	 When $(V, \rho^L, \rho^R, K_V)$ is the regular representation  $ (\frakg , [\cdot, \cdot]_\frakg,   K)$, we denote $\C^\bullet_{\RBO}(\frakg ,\frakg )$ by $\C^\bullet_{\RBO}(\frakg )$ and call it the cochain complex of  modified Rota-Baxter operator $K$, and denote $\rmH^\bullet_{\RBO}(\frakg ,\frakg )$ by $\rmH^\bullet_{\RBO}(\frakg )$ and call it the cohomology of  modified Rota-Baxter operator $K$.
 \end{defn}

\subsection{Cohomology of modified Rota-Baxter Leibniz algebras}\
\label{Subsec:chomology RB}

In this subsection, we will combine the  cohomology of   Leibniz algebras and the cohomology of modified Rota-Baxter operators to define a cohomology theory for modified Rota-Baxter Leibniz algebras.

Let $V=(V, \rho^L, \rho^R, K_V)$ be a representation over a modified Rota-Baxter Leibniz algebra   $\frakg =(\frakg ,\mu=[\cdot, \cdot]_\frakg,K)$. Now, let's construct a chain map   $$\Phi^\bullet:\C^\bullet_{\PLA}(\frakg ,V) \rightarrow C_{\RBO}^\bullet(\frakg ,V),$$ i.e., the following commutative diagram:
\[\xymatrix{
		\C^0_{\PLA}(\frakg ,V)\ar[r]^-{\delta^0}\ar[d]^-{\Phi^0}& \C^1_{\PLA}(\frakg,V)\ar@{.}[r]\ar[d]^-{\Phi^1}&\C^n_{\PLA}(\frakg,V)\ar[r]^-{\delta^n}\ar[d]^-{\Phi^n}&\C^{n+1}_{\PLA}(\frakg,V)\ar[d]^{\Phi^{n+1}}\ar@{.}[r]&\\
		\C^0_{\RBO}(\frakg,V)\ar[r]^-{\partial^0}&\C^1_{\RBO}(\frakg,V)\ar@{.}[r]& \C^n_{\RBO}(\frakg,V)\ar[r]^-{\partial^n}&\C^{n+1}_{\RBO}(\frakg,V)\ar@{.}[r]&
.}\]

Define $\Phi^0=\Id_{\Hom(k,V)}=\Id_V$, and for  $n\geqslant 1$ and $ f\in \C^n_{\PLA}(\frakg,V)$,  define $\Phi^n(f)\in \C^n_{\RBO}(\frakg,V)$ as:
\begin{align*}
  \Phi^n(f)(x_1,\ldots, x_{n}) =&f(K(x_1), \ldots, K(x_n))\\
 &-\sum_{1\leqslant i_1<\cdots< i_{r}\leqslant n,~~ r ~~\mathrm{odd}} (-\lambda)^{\frac{r-1}{2}} K_V\circ f (K(x_1), \ldots, x_{i_1}, \ldots, x_{i_2},\ldots, x_{i_r},\ldots, K(x_n)) \\
 & -\sum_{1\leqslant i_1< \cdots< i_{r}\leqslant n,~~ r ~~\mathrm{even}}  (-\lambda)^{\frac{r}{2}+1} K_V\circ f (K(x_1), \ldots, x_{i_1}, \ldots, x_{i_2},\ldots, x_{i_r},\ldots, K(x_n)).
\end{align*}

\smallskip

\begin{prop}\label{Prop: Chain map Phi}
	The map $\Phi^\bullet: \C^\bullet_\PLA(\frakg,V)\rightarrow \C^\bullet_{\RBO}(\frakg,V)$ is a chain map.
\end{prop}

%This result follows from the $L_\infty$-structure over the cochain complex of Rota-Baxter pre-Lie algebras, so we omit it; see Proposition~\ref{Prop: cohomlogy complex as underlying complex of L infinity algebra}.

\smallskip

%Multiplying $\Phi^n$ by $(-1)^n$, we can make the above commutative diagram into a bicomplex, denote it by $\C^{\bullet,\bullet}_{\RBA}(\frakg,M)$.
\begin{defn}
 Let $V=(V, \rho^L, \rho^R, K_V)$ be a  representation over a modified Rota-Baxter Leibniz algebra   $\frakg =(\frakg ,\mu=[\cdot, \cdot]_\frakg, K)$.  We define the  cochain complex $(\C^\bullet_{\RBA}(\frakg,V), d^\bullet)$  of  modified Rota-Baxter Leibniz algebras $(\frakg,\mu,K)$ with coefficients in $(V, \rho^L, \rho^R, K_V)$ to the negative shift of the mapping cone of $\Phi^\bullet$, that is,   let
\[\C^0_{\RBA}(\frakg,V)=\C^0_\PLA(\frakg,V)  \quad  \mathrm{and}\quad   \C^n_{\RBA}(\frakg,V)=\C^n_\PLA(\frakg,V)\oplus \C^{n-1}_{\RBO}(\frakg,V), \forall n\geqslant 1,\]
 and the differential $d^n: \C^n_{\RBA}(\frakg,V)\rightarrow \C^{n+1}_{\RBA}(\frakg,V)$ is given by \[d^n(f,g)= (\delta^n(f), -\partial^{n-1}(g)  -\Phi^n(f))\]
 for any $f\in \C^n_\PLA(\frakg,V)$ and $g\in \C^{n-1}_{\RBO}(\frakg,V)$.
The  cohomology of $(\C^\bullet_{\RBA}(\frakg,V), d^\bullet)$, denoted by $\rmH_{\RBA}^\bullet(\frakg,V)$,  is called the cohomology of the modified Rota-Baxter Leibniz algebra $(\frakg,\mu,K)$ with coefficients in $(V, \rho^L, \rho^R, K_V)$.
When $(V, \rho^L, \rho^R, K_V)=(\frakg,\mu, K)$, we just denote $\C^\bullet_{\RBA}(\frakg,\frakg), \rmH^\bullet_{\RBA}(\frakg,\frakg)$   by $\C^\bullet_{\RBA}(\frakg),  \rmH_{\RBA}^\bullet(\frakg)$ respectively, and call  them the cochain complex, the cohomology of  modified Rota-Baxter Leibniz algebra $(\frakg,\mu,T)$ respectively.
\end{defn}
There is an obvious short exact sequence of complexes:
\begin{eqnarray}\label{Seq of complexes} 0\to \C^\bullet_{\RBO}(\frakg,V)\to \C^\bullet_{\RBA}(\frakg,V)\to \C^\bullet_{\PLA}(\frakg,V)\to 0\end{eqnarray}
which induces a long exact sequence of cohomology groups
$$0\to \rmH^{0}_{\RBA}(\frakg, V)\to\mathrm{H}^0_{\mathrm{Leib}}(\frakg, V)\to\rmH^0_{\RBO}(\frakg, V) \to \rmH^{1}_{\RBA}(\frakg, V)\to\mathrm{H}^1_{\mathrm{Leib}}(\frakg, V)\to\pl s$$
$$\cdots\to \mathrm{H}^p_{\mathrm{Leib}}(\frakg, V)\to \rmH^p_{\RBO}(\frakg, V)\to \rmH^{p+1}_{\RBA}(\frakg, V)\to \mathrm{H}^{p+1}_{\mathrm{Leib}}(\frakg, V)\to \cdots$$

%It is ready to see that up to signs, the complex $(\C^\bullet_{\RA}(A,M), d^\bullet)$ is exactly  the shift of the mapping cone of $\Phi^\bullet$.

\bigskip

%It is ready to see that up to signs, the complex $(\C^\bullet_{\RB\frakg}(\frakg,M), d^\bullet)$ is exactly  the shift of the mapping cone of $\Phi^\bullet$.

\bigskip

\section{Formal deformations of modified Rota-Baxter Leibniz algebras}

In this section, we will study formal deformations of modified Rota-Baxter Leibniz algebras and interpret  them  via    lower degree   cohomology groups  of modified Rota-Baxter Leibniz algebras defined in last section.

Let $(\frakg,\mu, K)$ be a modified Rota-Baxter Leibniz  algebra.   Consider a 1-parameterized family:
\[\mu_t=\sum_{i=0}^\infty \mu_it^i, \ \mu_i\in \C^2_\PLA(\frakg),\quad  K_t=\sum_{i=0}^\infty K_it^i,  \ K_i\in \C^1_{\RBO}(\frakg).\]

\begin{defn}
	A  1-parameter formal deformation of   modified Rota-Baxter Leibniz algebra $(\frakg, \mu,K)$ is a pair $(\mu_t,K_t)$ which endows the flat $\bfk[[t]]$-module $\frakg[[t]]$ with a modified Rota-Baxter Leibniz algebra structure over $\bfk[[t]]$ such that $(\mu_0,K_0)=(\mu,K)$.
\end{defn}

 Power series $\mu_t$ and $ K_t$ determine a  1-parameter formal deformation of modified Rota-Baxter Leibniz algebra $(\frakg,\mu,K)$ if and only if for any $x, y, z\in \frakg$, the following equations hold :
\begin{align*}
&\mu_t(x,  \mu_t(y, z))= \mu_t(\mu_t(x, y), z)+\mu_t(y,\mu_t(x, z))\\
& \mu_t(K_t(x),  K_t(y))=K_t(\mu_t(K_t(x), y)+\mu_t(x, K_t(y))\big)+\lambda\mu_t(x, y).
 \end{align*}
By expanding these equations and comparing the coefficient of $t^n$, we obtain  that $\{\mu_i\}_{i\geqslant0}$ and $\{K_i\}_{i\geqslant0}$ have to  satisfy: for any $n\geqslant 0$,
\begin{align}\label{Eq: deform eq for  products in RBA}
&\sum_{i+j=n}\mu_i(x, \mu_j(y, z))= \sum_{i+j=n}\mu_i(\mu_j(x, y),z)+\mu_i(y,\mu_j(x, z)),\end{align}
\begin{align}\label{Eq: Deform RB operator in RBA}
&\sum_{i+j+k=n}\mu_i(K_j(x),  K_k(y))=\sum_{i+j+k=n}K_i\big(\mu_j(K_k(x),y)+\mu_j(x,  K_k(y))\big)+\lambda \mu_k(x, y).
\end{align}
Obviously, when $n=0$, the above conditions are exactly the associativity of $\mu=\mu_0$ and Equation~(\ref{Eq: Rota-Baxter relation}) which is the defining relation of a modified Rota-Baxter operator $K=K_0$.

\smallskip

\begin{prop}\label{Prop: Infinitesimal is 2-cocyle}
	Let $(\frakg[[t]],\mu_t,K_t)$ be a  1-parameter formal deformation of a modified Rota-Baxter Leibniz algebra $(\frakg,\mu,K)$. Then
	$(\mu_1,K_1)$ is a 1-cocycle in the cochain complex
	$C_{\RBA}^\bullet(\frakg)$.
\end{prop}
\begin{proof} When $n=1$,   Equations~(\ref{Eq: deform eq for  products in RBA}) and (\ref{Eq: Deform RB operator in RBA})  become
	 \begin{align*}
&[x,  \mu_1(y, z)]_\frakg+\mu_1(x, [y, z]_\frakg)=[\mu_1(x, y),z]_\frakg+[y,\mu_1(x, z)]_\frakg+\mu_1([x, y]_\frakg,z)+\mu_1(y,[x, z]_\frakg),
\end{align*}
and
\begin{align*}
&[K_1(x),  K(y)]_\frakg+[K(x), K_1(y)]_\frakg+\mu_1(K(x),  K(y))\\
&=K_1\big([K(x), y]_\frakg+[ x, K(y)]_\frakg\big)+\lambda[x, y]_\frakg+K\big(\mu_1(K(x), y)+\mu_1( x, K(y))\big)+\lambda[x, y]_\frakg\\
&+K\big([K_1(x), y]_\frakg+[ x, K_1(y)]_\frakg\big)+\lambda\mu_1(x, y).
\end{align*}
Note that  the first equation is exactly $\delta^2(\mu_1)=0\in \C^\bullet_{\PLA}(\frakg)$ and that  second equation is exactly to  \[\Phi^2(\mu_1)=-\partial^1(K_1) \in \C^\bullet_{\RBO}(\frakg).\]
	So $(\mu_1,K_1)$ is a 2-cocycle in $\C^\bullet_{\RBA}(\frakg)$.
	\end{proof}

\smallskip

\begin{defn} The 1-cocycle $(\mu_1,K_1)$ is called the infinitesimal of the 1-parameter formal deformation $(\frakg[[t]],\mu_t,K_t)$ of a modified Rota-Baxter Leibniz algebra $(\frakg,\mu,K)$.
	\end{defn}

\smallskip
\begin{defn}
Let $(\frakg[[t]],\mu_t,K_t)$ and $(\frakg[[t]],\mu_t',K_t')$ be two 1-parameter formal deformations of a modified Rota-Baxter Leibniz algebra $(\frakg,\mu,K)$. A formal isomorphism from $(\frakg[[t]],\mu_t',K_t')$ to $(\frakg[[t]], \mu_t, K_t)$ is a power series $\psi_t=\sum_{i=0}\psi_it^i: \frakg[[t]]\rightarrow \frakg[[t]]$, where $\psi_i: \frakg\rightarrow \frakg$ are linear maps with $\psi_0=\Id_\frakg$, such that:
\begin{eqnarray}\label{Eq: equivalent deformations}\psi_t\circ \mu_t' &=& \mu_t\circ (\psi_t\ot \psi_t),\\
\psi_t\circ K_t'&=&K_t\circ\psi_t. \label{Eq: equivalent deformations2}
	\end{eqnarray}
	In this case, we say that the two 1-parameter formal deformations $(\frakg[[t]], \mu_t,K_t)$ and
	$(\frakg[[t]],\mu_t',K_t')$ are  equivalent.
\end{defn}

\smallskip

Given a modified Rota-Baxter Leibniz  algebra $(\frakg,\mu,K)$, the power series $\mu_t,K_t$
with $\mu_i=\delta_{i,0}\mu, K_i=\delta_{i,0}K$ make
$(\frakg[[t]],\mu_t,K_t)$ into a $1$-parameter formal deformation of
$(\frakg,\mu,K)$. Formal deformations equivalent to this one are called trivial.
\smallskip

\begin{thm}
The infinitesimals of two equivalent 1-parameter formal deformations of $(\frakg,\mu,K)$ are in the same cohomology class in $\rmH^\bullet_{\RBA}(\frakg)$.
\end{thm}

\begin{proof} Let $\psi_t:(\frakg[[t]],\mu_t',K_t')\rightarrow (\frakg[[t]],\mu_t,K_t)$ be a formal isomorphism.
	Expanding the identities and collecting coefficients of $t$, we get from Equations~(\ref{Eq: equivalent deformations}) and (\ref{Eq: equivalent deformations2}):
	\begin{eqnarray*}
		\mu_1'&=&\mu_1+\mu\circ(\Id\ot \psi_1)-\psi_1\circ\mu+\mu\circ( \psi_1\ot\Id),\\
		K_1'&=&K_1+K\circ\psi_1-\psi_1\circ K,
		\end{eqnarray*}
	that is, we have\[(\mu_1',K_1')-(\mu_1,K_1)=(\delta^1(\psi_1), -\Phi^1(\psi_1))=d^1(\psi_1,0)\in  \C^\bullet_{\RBA}(\frakg).\]
\end{proof}

\smallskip

\begin{defn}
	A modified Rota-Baxter Leibniz algebra $(\frakg,\mu,K)$ is said to be rigid if every 1-parameter formal deformation is trivial.
\end{defn}

\begin{thm}
	Let $(\frakg,\mu,K)$ be a modified Rota-Baxter Leibniz algebra. If $\rmH^2_{\RBA}(\frakg)=0$, then $(\frakg,\mu,K)$ is rigid.
\end{thm}

\begin{proof}Let $(\frakg[[t]], \mu_t, K_t)$ be a $1$-parameter formal deformation of $(\frakg, \mu, K)$. By Proposition~\ref{Prop: Infinitesimal is 2-cocyle},
$(\mu_1, K_1)$ is a $2$-cocycle. By $\rmH^2_{\RBA}(\frakg)=0$, there exists a $1$-cochain $$(\psi_1', x) \in \C^1_\RBA(\frakg)= C^1_{\PLA}(\frakg)\oplus \Hom(k, \frakg)$$ such that
$(\mu_1, K_1) =  d^1(\psi_1', x), $
that is, $\mu_1=\delta^1(\psi_1')$ and $K_1=-\partial^0(x)-\Phi^1(\psi_1')$. Let $\psi_1=\psi_1'+\delta^0(x)$. Then
 $\mu_1= \delta^1(\psi_1)$ and $K_1=-\Phi^1(\psi_1)$, as it can be readily seen that $\Phi^1(\delta^0(x))=\partial^0(x)$.

Setting $\psi_t = \Id_\frakg -\psi_1t$, we have a deformation $(\frakg[[t]], \overline{\mu}_t, \overline{K}_t)$, where
$$\overline{\mu}_t=\psi_t^{-1}\circ \mu_t\circ (\psi_t\times \psi_t)$$
and $$\overline{K}_t=\psi_t^{-1}\circ K_t\circ \psi_t.$$
  It can be easily verify  that $\overline{\mu}_1=0, \overline{K}_1=0$. Then
    $$\begin{array}{rcl} \overline{\mu}_t&=& \mu+\overline{\mu}_2t^2+\cdots,\\
 K_t&=& K+\overline{K}_2t^2+\cdots.\end{array}$$
   By Equations~(\ref{Eq: deform eq for  products in RBA}) and (\ref{Eq: Deform RB operator in RBA}), we see that $(\overline{\mu}_2,  \overline{K}_2)$ is still a $2$-cocyle, so by induction, we can show that
  $ (\frakg[[t]], \mu_t , K_t) $ is equivalent to the trivial extension $(\frakg[[t]], \mu, K).$
Thus, $(\frakg,\mu,K)$ is rigid.

\end{proof}

\section{Abelian extensions of modified Rota-Baxter Leibniz algebras}

 In this section, we study abelian extensions of modified Rota-Baxter Leibniz algebras and show that they are classified by the first cohomology, as one would expect of a good cohomology theory.

 Notice that a vector space $V$ together with a linear transformation $K_V:V\to V$ is naturally a modified Rota-Baxter Leibniz algebra where the multiplication on $V$ is defined to be $[u, v]_V=0$ for all $u, v\in V.$

 \begin{defn}
 	An   abelian extension  of modified Rota-Baxter Leibniz algebras is a short exact sequence of  morphisms of modified Rota-Baxter Leibniz algebras
 \begin{eqnarray}\label{Eq: abelian extension} 0\to (V,[\cdot, \cdot]_V, K_V)\stackrel{i}{\to} (\hat{\frakg},[\cdot, \cdot]_{\hat{\frakg}},  \hat{K})\stackrel{p}{\to} (\frakg, [\cdot, \cdot]_{\frakg}, K)\to 0,
 \end{eqnarray}
 that is, there exists a commutative diagram:
 	\[\begin{CD}%\label{dia:ext}
 		0@>>> {V} @>i >> \hat{\frakg} @>p >> \frakg @>>>0\\
 		@. @V {K_V} VV @V {\hat{K}} VV @V K VV @.\\
 		0@>>> {V} @>i >> \hat{\frakg} @>p >> \frakg @>>>0,
 	\end{CD}\]
 where the modified Rota-Baxter Leibniz algebra $(V, [\cdot, \cdot]_V, K_V)$	satisfies  $[\cdot ,v]_V=0$ for all $v \in V.$

 We will call $(\hat{\frakg},[\cdot, \cdot]_{\hat{\frakg}}, \hat{K})$ an abelian extension of $(\frakg,[\cdot, \cdot]_{\frakg}, K)$ by $(V,[\cdot, \cdot]_V, K_V)$.
 \end{defn}

 \begin{defn}
 	Let $(\hat{\frakg}_1,[\cdot, \cdot]_{\hat{\frakg}_1}, \hat{K}_1)$ and $(\hat{\frakg}_2,[\cdot, \cdot]_{\hat{\frakg}_2}, \hat{K}_2)$ be two abelian extensions of $(\frakg,[\cdot, \cdot]_{\frakg}, K)$ by $(V,[\cdot, \cdot]_V, K_V)$. They are said to be  isomorphic  if there exists an isomorphism of modified Rota-Baxter Leibniz algebras $\zeta:(\hat{\frakg}_1,\hat{K}_1)\rar (\hat{\frakg}_2,\hat{K}_2)$ such that the following commutative diagram holds:
 	\begin{eqnarray}\label{Eq: isom of abelian extension}\begin{CD}%\label{dia:ext}
 		0@>>> {(V,[\cdot, \cdot]_V, K_V)} @>i >> (\hat{\frakg}_1,[\cdot, \cdot]_{\hat{\frakg}_1},{\hat{K}_1}) @>p >> (\frakg,[\cdot, \cdot]_{\frakg}, K) @>>>0\\
 		@. @| @V \zeta VV @| @.\\
 		0@>>> {(V,[\cdot, \cdot]_V, K_V)} @>i >> (\hat{\frakg}_2,[\cdot, \cdot]_{\hat{\frakg}_2},{\hat{K}_2}) @>p >> (\frakg,[\cdot, \cdot]_{\frakg}, K) @>>>0.
 	\end{CD}\end{eqnarray}
 \end{defn}

 A   section of an abelian extension $(\hat{\frakg},[\cdot, \cdot]_{\hat{\frakg}}, {\hat{K}})$ of $(\frakg,[\cdot, \cdot]_{\frakg}, K)$ by $(V,[\cdot, \cdot]_V, K_V)$ is a linear map $s:\frakg\rar \hat{\frakg}$ such that $p\circ s=\Id_\frakg$.

 We will show that isomorphism classes of  abelian extensions of $(\frakg,[\cdot, \cdot]_{\frakg}, K)$ by $(V,[\cdot, \cdot]_V, K_V)$ are in bijection with the second cohomology group   ${\rmH}_{\RBA}^2(\frakg,V)$.

 \bigskip

Let    $(\hat{\frakg},[\cdot, \cdot]_{\hat{\frakg}}, \hat{K})$ be  an abelian extension of $(\frakg,[\cdot, \cdot]_{\frakg}, K)$ by $(V,[\cdot, \cdot]_V, K_V)$ having the form Equation~\eqref{Eq: abelian extension}. Choose a section $s:\frakg\rar \hat{\frakg}$, and define $\rho^L,\rho^R: \frakg\rightarrow \mathfrak{gl}(V)$
 \begin{eqnarray*}
\rho^L(x)v:=[s(x), i(v)]_{\hat{\frakg}}~~~~\rho^R(x)v:=[i(v), s(x)]_{\hat{\frakg}}.
 \end{eqnarray*}
 for all $x  \in \frakg, v\in V$.
 \begin{prop}\label{Prop: new RB bimodules from abelian extensions}
 	With the above notations, $(V, \rho^L, \rho^R, K_V)$ is a representation over $(\frakg,[\cdot, \cdot]_\frakg, K)$.
 \end{prop}
 \begin{proof}
 	It is straightforward to check that $(V, \rho^L, \rho^R)$ is a representation over $(\frakg,[\cdot, \cdot]_\frakg)$. Moreover, ${\hat{K}}(s(x))-s(K(x))\in V$ means that  $\rho^L({\hat{K}}(s(x)) v=\rho^L(s(K(x))) v$ and $\rho^R({\hat{K}}(s(x)) v=\rho^R(s(K(x))) v$. Thus we have
 	\begin{align*}
 		\rho^L(K(x))K_V(v)
 &= \rho^L(sK(x))K_V(v)\\
 		&=K_V\big(\rho^L(\hat{K}(s(x))) v +\rho^L(s(x))K_V(v) \big)+\lambda\rho^L(s(x))v\\
 		&=K_V\big(\rho^L(K(x)) v +\rho^L(x)K_V(v) \big)+\lambda\rho^L(x)v.
 	\end{align*}
 Similar to prove that
 \begin{eqnarray*}
&\rho^R(K(x))K_V(v)&=K_V\big(\rho^R(K(x)) v +\rho^R(x)K_V(v)\big) +\lambda \rho^R(x)v.
 \end{eqnarray*}
 	Hence, $(V, \rho^L, \rho^R, K_V)$ is a  representation over $(\frakg, [\cdot, \cdot]_\frakg,  K)$.
 \end{proof}

 We  further  define linear maps $\psi:\frakg\otimes \frakg\rar V$ and $\chi:\frakg\rar V$ respectively by
 \begin{align*}
 	\psi(x, y)&=[s(x), s(y)]_{\hat{\frakg}}-s([x, y]_\frakg),\quad\forall x, y\in \frakg,\\
 	\chi(x)&={\hat{K}}(s(x))-s(K(x)),\quad\forall x\in \frakg.
 \end{align*}

 %We now  show that $(\psi, \chi)$ is a  2-cocycle  of the Rota-Baxter pre-Lie algebra $(A,T)$ with   coefficients  in $(M,T_M)$.
 \begin{prop}\label{prop:2-cocycle}
 	 The pair
 	$(\psi,\chi)$ is a 2-cocycle  of  modified Rota-Baxter Leibniz algebras $(\frakg,[\cdot, \cdot]_\frakg, K)$ with  coefficients  in the representation $(V, \rho^L, \rho^R, K_V)$ introduced in Proposition~\ref{Prop: new RB bimodules from abelian extensions}.
 \end{prop}

 	The proof is by direct computations, so it is left to the reader.

 The choice of the section $s$ in fact determines a splitting
 $$\xymatrix{0\ar[r]&  V\ar@<1ex>[r]^{i} &\hat{\frakg}\ar@<1ex>[r]^{p} \ar@<1ex>[l]^{t}& \frakg \ar@<1ex>[l]^{s} \ar[r] & 0}$$
 subject to $t\circ i=\Id_V, t\circ s=0$ and $ it+sp=\Id_{\hat{\frakg}}$.
 Then there is an induced isomorphism of vector spaces
 $$\left(\begin{array}{cc} p& t\end{array}\right): \hat{\frakg}\cong   \frakg\oplus V: \left(\begin{array}{c} s\\ i\end{array}\right).$$
We can  transfer the modified Rota-Baxter Leibniz algebra structure on $\hat{\frakg}$ to $\frakg\oplus V$ via this isomorphism.
  It is direct to verify that this  endows $\frakg\oplus V$ with a Leibniz bracket $[\cdot, \cdot]_\psi$ and a modified Rota-Baxter operator $K_\chi$ defined by
 \begin{align}
 	\label{eq:mul}[x+u, y+v]_\psi&=[x, y]_\frakg+\rho^L(x)v+\rho^R(y)u+\psi(x, y),\\
 	\label{eq:dif}K_\chi(x+u)&=K(x)+\chi(x)+K_V(u),
 \end{align}
 for any  $x, y\in \frakg,\,u, v\in V$.
 Moreover, we get an abelian extension
 $$0\to (V, [\cdot, \cdot]_V, K_V)\stackrel{\left(\begin{array}{cc} s& i\end{array}\right) }{\to} (\frakg\oplus V, K_\chi)\stackrel{\left(\begin{array}{c} p\\ t\end{array}\right)}{\to} (\frakg, [\cdot, \cdot]_\frakg, K)\to 0$$
 which is easily seen to be  isomorphic to the original one \eqref{Eq: abelian extension}.

 \medskip

 Now we investigate the influence of different choices of   sections.

 \begin{prop}\label{prop: different sections give}
 \begin{itemize}
 \item[(i)] Different choices of the section $s$ give the same  representation on $(V, [\cdot, \cdot]_V,  K_V)$;

 \item[(ii)]   the cohomological class of $(\psi,\chi)$ does not depend on the choice of sections.

 \end{itemize}

 \end{prop}
 \begin{proof}Let $s_1$ and $s_2$ be two distinct sections of $p$.
  We define $\gamma:\frakg\rar V$ by $\gamma(x)=s_1(x)-s_2(x)$.

  Since the modified Rota-Baxter  Leibniz algebra $(V, [\cdot, \cdot]_V,  K_V)$	satisfies  $[u, v]_V=0$ for all $u, v\in V$,
  $$\rho^L(s_1(x))u= \rho^L(s_2(x)+\gamma(x)) u=\rho^L(s_2(x) ) u.$$ So different choices of the section $s$ give the same  representation on $(V, [\cdot, \cdot]_V, K_V)$;

  We   show that the cohomological class of $(\psi,\chi)$ does not depend on the choice of sections.   Then
 	\begin{align*}
 		\psi_1(x, y)&=[s_1(x),  s_1(y)]_{\hat{\frakg}}-s_1([x, y]_\frakg)\\
 		&=[s_2(x)+\gamma(x), s_2(y)+\gamma(y)]_{\hat{\frakg}}-(s_2([x, y]_\frakg)+\gamma([x, y]_\frakg))\\
 		&=([s_2(x),  s_2(y)]_{\hat{\frakg}}-s_2([x, y]_\frakg))+[s_2(x), \gamma(y)]_{\hat{\frakg}}+[\gamma(x),  s_2(y)]_{\hat{\frakg}}-\gamma([x, y]_\frakg)\\
 		&=([s_2(x),  s_2(y)]_{\hat{\frakg}}-s_2([x, y]_\frakg))+[x, \gamma(y)]_{\hat{\frakg}}+[\gamma(x),  y]_{\hat{\frakg}}-\gamma([x, y]_\frakg)\\
 		&=\psi_2(x, y)+\delta(\gamma)(x, y)
 	\end{align*}
 	and
 	\begin{align*}
 		\chi_1(x)&={\hat{K}}(s_1(x))-s_1(K(x))\\
 		&={\hat{K}}(s_2(x)+\gamma(x))-(s_2(K(x))+\gamma(K(x)))\\
 		&=({\hat{K}}(s_2(x))-s_2(K(x)))+{\hat{K}}(\gamma(x))-\gamma(K(x))\\
 		&=\chi_2(x)+K_V(\gamma(x))-\gamma(K(x))\\
 		&=\chi_2(x)-\Phi^1(\gamma)(x).
 	\end{align*}
 	That is, $(\psi_1,\chi_1)=(\psi_2,\chi_2)+d^1(\gamma)$. Thus $(\psi_1,\chi_1)$ and $(\psi_2,\chi_2)$ form the same cohomological class  {in $\rmH_{\RBA}^2(\frakg,V)$}.

 \end{proof}

 We show now the isomorphic abelian extensions give rise to the same cohomology classes.
 \begin{prop}Let $V$ be a vector space and  $K_V\in\End_\bfk(V)$. Then $(V,[\cdot, \cdot]_V,  K_V)$ is a  modified Rota-Baxter Leibniz algebra  with trivial Leibniz bracket.
 Let $(\frakg,[\cdot, \cdot]_\frakg, K)$ be a  modified Rota-Baxter Leibniz algebra.
 Two isomorphic abelian extensions of modified Rota-Baxter Leibniz algebras $(\frakg,[\cdot, \cdot]_\frakg,  K)$ by  $(V, [\cdot, \cdot]_V,  K_V)$  give rise to the same cohomology class  in $\rmH_{\RBA}^2(\frakg,V)$.
 \end{prop}
 \begin{proof}
  Assume that $(\hat{\frakg}_1, [\cdot, \cdot]_{\hat{\frakg}_1},  {\hat{K}_1})$ and $(\hat{\frakg}_2, [\cdot, \cdot]_{\hat{\frakg}_2}, {\hat{K}_2})$ are two isomorphic abelian extensions of $(\frakg,[\cdot, \cdot]_{\frakg}, K)$ by $(V,[\cdot, \cdot]_V, K_V)$ as is given in \eqref{Eq: isom of abelian extension}. Let $s_1$ be a section of $(\hat{\frakg}_1,[\cdot, \cdot]_{\hat{\frakg}_1},  {\hat{K}_1})$. As $p_2\circ\zeta=p_1$, we have
 	\[p_2\circ(\zeta\circ s_1)=p_1\circ s_1=\Id_{\frakg}.\]
 	Therefore, $\zeta\circ s_1$ is a section of $(\hat{\frakg}_2,[\cdot, \cdot]_{\hat{\frakg}_2}, {\hat{K}_2})$. Denote $s_2:=\zeta\circ s_1$. Since $\zeta$ is a homomorphism of modified Rota-Baxter Leibniz algebras such that $\zeta|_V=\mathrm{id}_V$, $\zeta(\rho^L(x) m)=\zeta(\rho^L(s_1(x))m)=\rho^L(s_2(x)) m=\rho^L(x )m$, so $\zeta|_V: V\to V$ is compatible with the induced  representation.
 We have
 	\begin{align*}
 		\psi_2(x, y)&=[s_2(x), s_2(y)]_{\hat{\frakg}_2}-s_2([x, y]_{\frakg}\\
 &=[\zeta(s_1(x)), \zeta(s_1(y))]_{\hat{\frakg}_2}-\zeta(s_1([x, y]_{\frakg})\\
 		&=\zeta([s_1(x), s_1(y)]_{\hat{\frakg}_1}-s_1([x, y]_{\frakg})\\
 &=\zeta(\psi_1(x, y))\\
 		&=\psi_1(x, y)
 	\end{align*}
 	and
 	\begin{align*}
 		\chi_2(x)&={\hat{K}_2}(s_2(x))-s_2(K(x))={\hat{K}_2}(\zeta(s_1(x)))-\zeta(s_1(K(x)))\\
 		&=\zeta({\hat{K}_1}(s_1(x))-s_1(K(x)))=\zeta(\chi_1(x))\\
 		&=\chi_1(x).
 	\end{align*}
 	Consequently, two isomorphic abelian extensions give rise to the same element in {$\rmH_{\RBA}^2(\frakg,V)$}.
\end{proof}
 \bigskip

 Now we consider the reverse direction.

 Let $(V, \rho^L, \rho^R, K_V)$ be a representation over a modified Rota-Baxter Leibniz algebra $(\frakg, [\cdot, \cdot]_\frakg,  K)$, given two linear maps  $\psi:\frakg\otimes \frakg\rar V$ and $\chi:\frakg\rar V$, one can define  a Leibniz bracket $[\cdot, \cdot]_\psi$ and an operator $K_\chi$  on  $\frakg\oplus V$ by Equations~(\ref{eq:mul}) and (\ref{eq:dif}).
 The following fact is important:
 \begin{prop}\label{prop:2-cocycle}
 	The triple $(\frakg\oplus V,[\cdot, \cdot]_\psi, K_\chi)$ is a modified Rota-Baxter Leibniz algebra   if and only if
 	$(\psi,\chi)$ is a 2-cocycle  of the modified Rota-Baxter Leibniz algebra  $(\frakg,[\cdot, \cdot]_\frakg, K)$ with  coefficients  in $(V, \rho^L, \rho^R, K_V)$.
 In this case,    we obtain an abelian extension
  $$0\to (V, \rho^L, \rho^R, K_V)\stackrel{\left(\begin{array}{cc} 0& \mathrm{id}  \end{array}\right) }{\to} (\frakg\oplus V, K_\chi)\stackrel{\left(\begin{array}{c} \mathrm{id}\\ 0\end{array}\right)}{\to} (\frakg,[\cdot, \cdot]_\frakg,  K)\to 0,$$
  and the canonical section $s=\left(\begin{array}{cc}   \mathrm{id}  & 0\end{array}\right): (\frakg, K)\to (\frakg\oplus V, K_\chi)$ endows $V$ with the original representation.
 \end{prop}
 \begin{proof}
 	If $(\frakg\oplus V,[\cdot, \cdot]_\psi, K_\chi)$ is a  modified Rota-Baxter Leibniz algebra, then $[\cdot, \cdot]_\psi$ implies
 	\begin{align}\label{eq:mc}
 		&\psi(x, [y, z]_\frakg) + \rho^L(x)\psi(y, z)=\psi([x, y]_\frakg, z)+\rho^L(y)\psi(x,  z)+\rho^R(z)\psi(x, y),
 	\end{align}
 	which means $\delta^2(\psi)=0$ in $\C^\bullet(\frakg,V)$.
 	Since $K_\chi$ is a modified Rota-Baxter operator,
 	for any $x, y\in \frakg, u, v\in V$, we have
 	\begin{eqnarray*}
&&[K_\chi(x+u),  K_\chi(y+v)]_\psi=K_\chi\big([K_\chi(x+u), y+v]_\psi+[x+u, K_\chi(y+v)]_\psi\big)+\lambda [x+u, y+v]_\psi.\end{eqnarray*}

 	That is,
 	\[ \partial^1(\chi)+\Phi^2(\psi)=0.\]
 	Hence, $(\psi,\chi)$ is a  2-cocycle.
 	
 	 Conversely, if $(\psi,\chi)$ is a 2-cocycle, one can easily check that $(\frakg\oplus V,[\cdot, \cdot]_\psi, K_\chi)$ is a modified Rota-Baxter Leibniz algebra.

  The last statement is clear.
 \end{proof}

 Finally, we show the following result:
 \begin{prop}
 	Two cohomologous $2$-cocyles give rise to isomorphic abelian extensions.
 \end{prop}
 \begin{proof}

 	  Given two 2-cocycles $(\psi_1,\chi_1)$ and $(\psi_2,\chi_2)$, we can construct two abelian extensions $(\frakg\oplus V,[\cdot, \cdot]_{\psi_1},K_{\chi_1})$ and  $(\frakg\oplus V,[\cdot, \cdot]_{\psi_2},K_{\chi_2})$ via Equations~\eqref{eq:mul} and \eqref{eq:dif}. If they represent the same cohomology  class {in $\rmH_{\RBA}^2(\frakg,V)$}, then there exists two linear maps $\gamma_0:k\rightarrow V, \gamma_1:\frakg\to V$ such that $$(\psi_1,\chi_1)=(\psi_2,\chi_2)+(\delta^1(\gamma_1),-\Phi^1(\gamma_1)-\partial^0(\gamma_0)).$$
 	Notice that $\partial^0=\Phi^1\circ\delta^0$. Define $\gamma: \frakg\rightarrow V$ to be $\gamma_1+\delta^0(\gamma_0)$. Then $\gamma$ satisfies
 	\[(\psi_1,\chi_1)=(\psi_2,\chi_2)+(\delta^1(\gamma),-\Phi^1(\gamma)).\]
 	
 	Define $\zeta:\frakg\oplus V\rar \frakg\oplus V$ by
 	\[\zeta(x+u):=x-\gamma(x)+u.\]
 	Then $\zeta$ is an isomorphism of these two abelian extensions $(\frakg\oplus V,[\cdot, \cdot]_{\psi_1},K_{\chi_1})$ and  $(\frakg\oplus N,[\cdot, \cdot]_{\psi_2},K_{\chi_2})$.
 \end{proof}

\bigskip

\noindent
{{\bf Acknowledgments.} The work is supported by Natural Science Foundation of China (Grant Nos. 11871301, 12271292).

\end{document}